\documentclass[12pt]{amsart}
\usepackage{amscd,amsmath,amsthm,amssymb,enumerate}
\usepackage[left]{lineno}
\usepackage{pstricks}
\usepackage[utf8]{inputenc}
\usepackage{epstopdf}

\usepackage{comment}

\usepackage[all]{xy}

 \def\G{{\mathcal G}}

 \def\opn#1#2{\def#1{\operatorname{#2}}} 
 \opn\chara{char} \opn\length{\ell} \opn\pd{pd} \opn\rk{rk}
 \opn\projdim{proj\,dim} \opn\injdim{inj\,dim} \opn\rank{rank}
 \opn\depth{depth} \opn\grade{grade} \opn\height{height}
 \opn\embdim{emb\,dim} \opn\codim{codim}
 
 \opn\Tr{Tr} \opn\bigrank{big\,rank}
 \opn\superheight{superheight}\opn\lcm{lcm}
 \opn\trdeg{tr\,deg}
 \opn\reg{reg} \opn\lreg{lreg} \opn\ini{in} \opn\lpd{lpd}
 \opn\size{size} \opn\sdepth{sdepth}
 \opn\link{link}\opn\fdepth{fdepth}\opn\lex{lex}

 \opn\div{div} \opn\Div{Div} \opn\cl{cl} \opn\Cl{Cl}
 \opn\Spec{Spec} \opn\Supp{Supp} \opn\supp{supp} \opn\Sing{Sing}
 \opn\Ass{Ass} \opn\Min{Min}\opn\Mon{Mon}
 \opn\Ann{Ann} \opn\Rad{Rad} \opn\Soc{Soc}
 \opn\Im{Im} \opn\Ker{Ker} \opn\Coker{Coker} \opn\Am{Am}
 \opn\Hom{Hom} \opn\Tor{Tor} \opn\Ext{Ext} \opn\End{End}
 \opn\Aut{Aut} \opn\id{id}
 
 \opn\nat{nat}
 \opn\pff{pf}
 \opn\Pf{Pf} \opn\GL{GL} \opn\SL{SL} \opn\mod{mod} \opn\ord{ord}
 \opn\Gin{Gin} \opn\Hilb{Hilb}\opn\sort{sort}
 \opn\aff{aff} \opn
 \con{conv} \opn\relint{relint} \opn\st{st}
 \opn\lk{lk} \opn\cn{cn} \opn\core{core} \opn\vol{vol}  \opn\inp{inp} \opn\nilpot{nilpot}
 \opn\link{link} \opn\star{star}\opn\lex{lex}\opn\set{set}
 \opn\width{wd}
 \opn\ecart{ecart}
 \opn\gr{gr}
 
 \def\pot#1#2{#1[\kern-0.28ex[#2]\kern-0.28ex]}

 \opn\dirlim{\underrightarrow{\lim}}
 \opn\inivlim{\underleftarrow{\lim}}
 \let\to=\rightarrow
 
 \def\Implies{\ifmmode\Longrightarrow \else
         \unskip${}\Longrightarrow{}$\ignorespaces\fi}
 \def\implies{\ifmmode\Rightarrow \else
         \unskip${}\Rightarrow{}$\ignorespaces\fi}
 \def\iff{\ifmmode\Longleftrightarrow \else
         \unskip${}\Longleftrightarrow{}$\ignorespaces\fi}

 \let\:=\colon

 \def\Soc{{\mathbf Soc}}

 \def\opn#1#2{\def#1{\operatorname{#2}}} 
 \opn\chara{char} \opn\length{\ell} \opn\pd{pd} \opn\rk{rk}
 \opn\projdim{proj\,dim} \opn\injdim{inj\,dim} \opn\rank{rank}
 \opn\depth{depth} \opn\grade{grade} \opn\height{height}
 \opn\bigheight{bigheight}
 \opn\embdim{emb\,dim} \opn\codim{codim}
 
 \opn\superheight{superheight}\opn\lcm{lcm}
 \opn\trdeg{tr\,deg}
 \opn\reg{reg} \opn\lreg{lreg} \opn\ini{in} \opn\lpd{lpd}
 \opn\size{size} \opn\sdepth{sdepth}
 \opn\link{link}\opn\fdepth{fdepth}\opn\lex{lex}
 \opn\type{type}
 \opn\gap{gap}
 \opn\arithdeg{arith-deg}
 \opn\Deg{Deg}
 \opn\sat{sat}
 \opn\mat{mat}
 \opn\Mat{Mat}
 \opn\div{div} \opn\Div{Div} \opn\cl{cl} \opn\Cl{Cl}
 \opn\Spec{Spec} \opn\Supp{Supp} \opn\supp{supp} \opn\Sing{Sing}
 \opn\Ass{Ass} \opn\Min{Min}\opn\Mon{Mon} \opn\Max{Max}
 \opn\Ann{Ann} \opn\Rad{Rad} \opn\Soc{Soc}
 \opn\Im{Im} \opn\Ker{Ker} \opn\Coker{Coker} \opn\Am{Am}
 \opn\Hom{Hom} \opn\Tor{Tor} \opn\Ext{Ext} \opn\End{End}
 \opn\Aut{Aut} \opn\id{id}
 
 \opn\nat{nat}
 \opn\pff{pf}
 \opn\Pf{Pf} \opn\GL{GL} \opn\SL{SL} \opn\mod{mod} \opn\ord{ord}
 \opn\Gin{Gin} \opn\Hilb{Hilb}\opn\sort{sort}
 \opn\PF{PF}\opn\Ap{Ap}
 \opn\mult{mult}
 \opn\bight{bight}
 \opn\aff{aff}
 \opn\relint{relint} \opn\st{st}
 \opn\lk{lk} \opn\cn{cn} \opn\core{core} \opn\vol{vol}  \opn\inp{inp} \opn\nilpot{nilpot}
 \opn\link{link} \opn\star{star}\opn\lex{lex}\opn\set{set}
 \opn\width{wd}
 \opn\Fr{F}
 \opn\QF{QF}
 \opn\G{G}
 \opn\type{type}\opn\res{res}
 \opn\conv{conv}
 \opn\Shad{Shad}
 \opn\gr{gr}
 
 \def\pot#1#2{#1[\kern-0.28ex[#2]\kern-0.28ex]}

 \opn\dirlim{\underrightarrow{\lim}}
 \opn\inivlim{\underleftarrow{\lim}}
 \let\to=\rightarrow
 
 \def\Implies{\ifmmode\Longrightarrow \else
         \unskip${}\Longrightarrow{}$\ignorespaces\fi}
 \def\implies{\ifmmode\Rightarrow \else
         \unskip${}\Rightarrow{}$\ignorespaces\fi}
 \def\iff{\ifmmode\Longleftrightarrow \else
         \unskip${}\Longleftrightarrow{}$\ignorespaces\fi}

 \let\:=\colon
\theoremstyle{plain}
\newtheorem{theorem}{Theorem}[section]

\newtheorem{corollary}[theorem]{Corollary}

\newtheorem{lemma}[theorem]{Lemma}

\theoremstyle{definition}

\newtheorem{example}[theorem]{Example}

\newtheorem{ex}[theorem]{Example}

\newtheorem{remark}[theorem]{Remark}

 \let\epsilon\varepsilon
 \let\kappa=\varkappa
 \def\qed{\ifhmode\textqed\fi
       \ifmmode\ifinner\quad\qedsymbol\else\dispqed\fi\fi}
 \def\textqed{\unskip\nobreak\penalty50
        \hskip2em\hbox{}\nobreak\hfil\qedsymbol
        \parfillskip=0pt \finalhyphendemerits=0}
 \def\dispqed{\rlap{\qquad\qedsymbol}}
 
 \opn\dis{dis}
 \def\pnt{{\raise0.5mm\hbox{\large\bf.}}}
 
 \opn\Lex{Lex}

\newcommand{\rme}{\mathrm{e}}
\newcommand{\rmf}{\mathrm{f}}

\newcommand{\fka}{\mathfrak{a}}
\newcommand{\fkb}{\mathfrak{b}}

\newcommand{\fkm}{\mathfrak{m}}

\newcommand{\m}{\mathfrak{m}}

\title{The irreducible multiplicity and Ulrich modules}

\author{Tran Nguyen An}
\address{Tran Nguyen An: Thai Nguyen University of Education, Vietnam}
\email{antrannguyen@gmail.com}

\author{Shinya Kumashiro}
\address{Shinya Kumashiro: National Institute of Technology (KOSEN), Oyama College
771 Nakakuki, Oyama, Tochigi, 323-0806, Japan}
\email{skumashiro@oyama-ct.ac.jp}

\thanks{This first author was partially supported by Vietnam National Foundation for Science
and Technology Development (NAFOSTED) under grant number 101.04-2019.309.
The second author was supported by JSPS KAKENHI Grant Number 21K13766.}

\begin{document}

\begin{abstract}
In this paper, we give a relation between the Hilbert multiplicity and the irreducible multiplicity. As an application, we characterize Ulrich modules in term of the irreducible multiplicity.
\end{abstract}

\maketitle


\section{Introduction}\label{section1}

The purpose of this paper is to study the irreducible multiplicity in connection with the (Hilbert) multiplicity. Let $(R, \m)$ be a Noetherian local ring, $I$ an $\m$-primary ideal, and $M$ a finitely generated $R$-module of dimension $t$. Then, it is well known that $\ell_R(M/I^{n+1}M)$ agrees with a polynomial function of degree $t$ for $n\gg 0$. That is, there exist integers $\rme_I^0(M), \dots, \rme_I^{t}(M)$ such that
\[
\ell_R(M/I^{n+1}M)=\rme_I^0(M)\binom{n+t}{t}-\rme_I^1(M)\binom{n+t-1}{n-1}+ \cdots +(-1)^{t}\rme_I^{t}(M)
\]
for $n\gg0$. 
$\rme_I^0(M), \dots, \rme_I^{t}(M)$ are called the {\it Hilbert coefficients} of $M$ with respect to $I$. The leading coefficient $\rme_I^0(M)$ is called the {\it (Hilbert) multiplicity} of $M$ with respect to $I$. The Hilbert multiplicity/coefficients are deeply studied in connection with the structure of associated graded rings and that of $M$ (see for examples \cite{GNO, GNO2, K, No, O, S}). 

On the other hand, it is also known that there exist integers $\rmf_I^0(M), \dots, \rmf_I^{t-1}(M)$ such that
\begin{align} \label{eq00}
\begin{split} 
& \ell_R(I^{n+1}M:_M \fkm/I^{n+1}M) \\
= & \rmf_I^0(M)\binom{n+t-1}{t-1}-\rmf_I^1(M)\binom{n+t-2}{n-2}+ \cdots +(-1)^{t-1}\rmf_I^{t-1}(M)
\end{split}
\end{align}
for $n\gg 0$ (\cite[Theorem 4.1]{CQT}). Note that the function $\ell_R(I^{n+1}M:_M \fkm/I^{n+1}M)$ is useful to study the index of reducibility. Here, a submodule $N$ of $M$ is called an {\it irreducible submodule} if $N$ cannot be written as an intersection of two properly larger submodules of $M$.  The number of irreducible components of an irredundant irreducible decomposition of $N$, which is independent of the choice of the decomposition by Noether \cite{Noe}, is called the {\it index of reducibility} of $N$ and denoted by $\mathrm{ir}_M(N)$. The study of the index of reducibility has attracted the interest of a number of researchers \cite{AK, ADKN, CQT, EN, GSu, Nor, Q1, Q2, T1, T2}. By noting that $\mathrm{ir}_M(N) = \ell_R((N:_M \frak m)/N)$ holds for a submodule $N$ of $M$ with $\ell_R(M/N) < \infty$, the equation \eqref{eq00} claims that the function $\mathrm{ir}_M(I^{n+1}M)$ agrees with a polynomial function of degree $t-1$ for $n\gg 0$. 
We call $\rmf_I^0(M), \dots, \rmf_I^{t-1}(M)$ the {\it irreducibility coefficients} of $M$ with respect to $I$. The leading coefficient $\rmf_I^0(M)$ is called the {\it irreducible multiplicity} of $M$ with respect to $I$ (see \cite{T2}). 

In light of the above results, the notions of irreducibility function is useful, but the behavior of irreducibility function is more mysterious than that of the Hilbert function. Indeed, in contrast to the Hilbert multiplicity, the additive formula does not hold for the irreducible multiplicity because the socle is not additive. With this perspective, this paper deals with the relationship between the Hilbert multiplicity $\rme_I^0(M)$ and the irreducible multiplicity $\rmf_I^0(M)$. Our result of this paper can be stated as follows.

\begin{theorem}\label{mainthm1}
Let $(R, \m)$ be a Noetherian local ring, $I$ an $\m$-primary ideal, and $M$ a finitely generated $R$-module of dimension $t$. Then,
\[
\rmf_I^0(M) \le 
\begin{cases} 
\rme_I^0(M) & \text{if $t\ne 1$}\\
\rme_I^0(M) + \ell_R((0):_M \fkm) & \text{if $t=1$.}
\end{cases}
\]
Furthermore, we have the equality  if $\fkm =I^{n+1}M:_R I^n M$ for some $n\ge 0$.
\end{theorem}

Note that the inequality $\rmf_I^0(M) \le \rme_I^0(M)$ does not directly follow from the inequality $\ell_R(I^{n+1}M:_M~\fkm/I^{n+1}M) \le \ell_R(M/I^{n+1}M)$ because the degrees of polynomials are different. Indeed, $\rmf_I^0(M) > \rme_I^0(M)$ happen when $t=1$ (see Example \ref{ex}). 
In addition, Theorem \ref{mainthm1} gives a characterization of Ulrich modules. Recall that an $R$-module $M$ is called an {\it Ulrich} $R$-module if $M$ is a Cohen-Macaulay $R$-module and $\mu_R(M)=\rme_\fkm^0(M)$, where $\mu_R(M)$ denotes the number of minimal generators of $M$ (\cite{BHU, GOTWY}). 
By recalling that an $R$-module $M$ is a Cohen-Macaulay $R$-module if and only if $\rme_{Q}^0(M)=\ell(M/Q M)$ for any (for some) parameter ideal $Q$ of $M$, the following is an analogue of the result in case of the irreducibility multiplicity.

\begin{corollary}\label{ulrich}
Let $M$ be a finitely generated $R$-module of dimension $t\ne 1$. Then the following are equivalent:
\begin{enumerate}[{\rm (a)}] 
\item $M$ is an Ulrich $R$-module;
\item $\rmf_Q^0(M) =\ell_R(M/QM)$ for some parameter ideal $Q$.
\end{enumerate} 
\end{corollary}

\section{The proof of main result}
In what follows, let $(R, \m)$ be a Noetherian local ring, $I$ an ideal, and $M$ a finitely generated $R$-module of dimension $t$. 
First, we note a lemma to prove Theorem \ref{mainthm1}. This result was proved by \cite[Proposition 2.1]{Schenzel} for rings, but it is easy to extend to modules. 

\begin{lemma}\label{l1}  
Let $I, J$ be ideals of $R$. Then there exists a positive integer $k$ such that 
\[
I^{n+k}M:_{M}J=I^{n}(I^{k}M:_{M}J)+(0):_{M}J
\] 
for all $n\ge 1$.
\end{lemma}

Now we prove Theorem \ref{mainthm1}. Let $I$ be an $\m$-primary ideal of $R$.

\begin{proof}[Proof of Theorem \ref{mainthm1}] By Lemma \ref{l1} and the hypothesis that $I$ is  $\m$-primary, we can choose an integer $\ell>0$ such that
\[
I^{n+1}M:_M \fkm=I^{n+1-\ell}(I^\ell M:_M \fkm) + (0):_M \fkm \quad \text{and} \quad ((0):_M \fkm) \cap I^n M=0
\] 
for all $n\ge \ell$. Hence, 
\begin{align*} 
(I^{n+1}M:_M \fkm) \cap I^n M =& [I^{n+1-\ell}(I^\ell M:_M \fkm) + (0):_M \fkm] \cap I^n M\\
=&I^{n+1-\ell}(I^\ell M:_M \fkm) + ((0):_M \fkm) \cap I^n M\\
=&I^{n+1-\ell}(I^\ell M:_M \fkm) 
\end{align*}
 for all $n\ge \ell$. It follows that
 \begin{align*}
 &(I^{n+1}M:_M  \fkm)/[(I^{n+1}M:_M \fkm)  \cap I^n M] \\
 \cong&  [I^{n+1-\ell}(I^\ell M:_M \fkm) + (0):_M \fkm]/ I^{n+1-\ell}(I^\ell M:_M \fkm)\\
\cong & (0):_M \fkm.
\end{align*}
Therefore, we obtain that
{\small
 \begin{align}\label{eq1}
 \begin{split} 
& \ell_R((I^{n+1}M:_M  \fkm)/I^{n+1} M)\\
=& \ell_R([(I^{n+1}M:_M \fkm)\cap I^n M]/I^{n+1}M) + \ell_R((I^{n+1}M:_M  \fkm)/[(I^{n+1}M:_M \fkm)  \cap I^n M])\\
 =&\ell_R([(I^{n+1}M:_M \fkm)\cap I^n M]/I^{n+1}M) + \ell_R((0):_M \fkm)\\
 = & \ell_R (I^n M/I^{n+1}M) + \ell_R((0):_M \fkm) - \ell_R(I^nM/[(I^{n+1}M:_M \fkm)\cap I^n M])\\
 \le & \ell_R (I^n M/I^{n+1}M) + \ell_R((0):_M \fkm).
 \end{split}
 \end{align}
 }
 Since $\ell_R((0):_M \fkm)$ is constant, by comparing the leading coefficients, we get that $\rmf_I^0(M)\le \rme_I^0(M)$ if $t\ge 2$. If $t=1$, then we have $\rmf_I^0(M) \le \rme_I^0(M) + \ell_R((0):_M \fkm)$. For the case where $t=0$, we immediately get the inequality $$\rmf_I^0(M) =\ell_R ((0):_M \fkm)\le \ell_R(M)=\rme_I^0(M).$$
 
Suppose that $\fkm =I^{n+1}M:_R I^n M$ for some (all) $n\gg 0$. Then, $I^n M\subseteq I^{n+1}M:_M~\fkm$. This follows that the inequality in \eqref{eq1} becomes an equality.
\end{proof}

\begin{corollary}  
Let $(R, \m)$ be a Noetherian local ring, and $M$ a finitely generated $R$-module of dimension $t$. Then, 
\[
\rmf_\fkm^0(M) =
\begin{cases} 
\rme_\fkm^0(M) & \text{if $t\ne 1$}\\
\rme_\fkm^0(M) + \ell_R((0):_M \fkm) & \text{if $t=1$.}
\end{cases}
\]
\end{corollary}

\begin{proof}
Note that $\fkm$ satisfies the condition $\fkm =\fkm^{n+1}M:_R~\fkm^n M$ for all $n>0$. Hence, the assertion follows from Theorem \ref{mainthm1}.
\end{proof}

\begin{ex}\label{ex}
Let 
\[
R=K[[x_1, \dots , x_d, y_1, \dots, y_\ell]]/[(x_1, \dots , x_d)(y_1, \dots, y_\ell) + (y_1, \dots, y_\ell)^2],
\] 
where $K[[x_1, \dots , x_d, y_1, \dots, y_\ell]]$ denotes the formal power series ring over a field $K$. Let $\fkm$ denote the maximal ideal of $R$. Then, 
\begin{align*}
\ell_R(R/\fkm^{n+1}) &= \binom{n+d}{d} + \ell \text{\quad and}\\
\ell_R((\fkm^{n+1}:_R \fkm)/\fkm^{n+1}) &= \binom{n+d-1}{d-1} +\ell.
\end{align*}
Hence, $\rme_\fkm^0(R)=\rmf_\fkm^0(R)=1$ if $t\ne 1$, but $\rme_\fkm^0(R)=1 < \rmf_\fkm^0(R)=1+\ell$ if $t=1$.
\end{ex}

\begin{proof}
Let $\underline{X}$ and $\underline{Y}$ denote the sequences $x_1, \dots , x_d$ and $y_1, \dots, y_\ell$, respectively. By noting that 
\[
R/\fkm^{n+1}\cong K[[\underline{X}, \underline{Y}]]/[(\underline{X})^{n+1} + (\underline{X})(\underline{Y}) + (\underline{Y})^2],
\] 
$R/\fkm^{n+1}$ is spanned by the basis of $K[[\underline{X}]]/(\underline{X})^{n+1}$ and $\underline{Y}$ as a $K$-vector space. On the other hand, $(\fkm^{n+1}:_R \fkm)/\fkm^{n+1}$ is spanned by the socle of $K[[\underline{X}]]/(\underline{X})^{n+1}$ and $\underline{Y}$ as a $K$-vector space; hence, the assertion holds true.
\end{proof}

\begin{proof}[Proof of Corollary \ref{ulrich}]
(a) $\Rightarrow$ (b): We may assume that $R/\fkm$ is infinite. Then we have $\rmf_Q^0(M)=\rme_Q^0(M)$ because $\fkm M=QM$ for some parameter ideal $Q$. Furthermore, since $M$ is a Cohen-Macaulay $R$-module, we get $\rme_Q^0(M)=\ell_R(M/QM)$.

(b) $\Rightarrow$ (a): We may also assume that $R/\fkm$ is infinite. If $t\ne 1$, then $\rmf_Q^0(M) \le \rme_Q^0(M) \le \ell_R(M/QM)$. The equality $\rme_Q^0(M) =\ell_R(M/QM)$ forces $M$ to be a Cohen-Macaulay $R$-module. We then obtain that $\rmf_Q^0(M) =\ell_R ((QM:_M\fkm)/QM)$ by \cite[Theorem 5.2]{CQT}. It follows that $QM:_M \fkm =M$; hence, $\fkm M =QM$ holds as desired.
\end{proof}

The assertion of Corollary \ref{ulrich} does not hold if $t=1$.

\begin{ex}
Let $R=K[[x]]$ be a formal power series ring over a field $K$, and let $M=R\oplus K$. Then, 
\begin{align*}
\ell_R(M/xM)=\ell_R(R/xR) + \ell_R(K/xK)=\ell_R(K) + \ell_R(K) = 2.
\end{align*}
On the other hand, since $M/x^{n+1}M \cong R/x^{n+1}R \oplus K$, we have 
\[
\rmf_{(x)}^0(M)=\ell_R((x^{n+1}M:_M\fkm)/x^{n+1}M)=2.
\] 
Hence, $\rmf_{(x)}^0(M) = \ell_R(M/xM) =2$, but $M$ is not an Ulrich $R$-module.
\end{ex}

\end{document}